\documentclass[12pt]{amsart}

\usepackage{amsfonts,amsmath,amssymb,amsthm}
\usepackage{mathrsfs}
\usepackage{color}

\usepackage[colorlinks,backref]{hyperref}

\newtheorem{thm}{Theorem}[section]

\newtheorem{prp}[thm]{Proposition}
\newtheorem{cor}[thm]{Corollary}
\theoremstyle{definition}
\newtheorem{dfn}[thm]{Definition}
\theoremstyle{remark}
\newtheorem{rem}[thm]{Remark}

\newtheorem{notation}[thm]{Notation}

\numberwithin{equation}{section}

\DeclareMathOperator{\Ric}{Ric}

 \DeclareMathOperator{\trace}{tr}

\DeclareMathOperator{\sign}{sign}
\DeclareMathOperator{\diver}{div}

\newcommand{\2}{``}

\begin{document}
\thispagestyle{empty}
\title[Implications of Energy Conditions]
{Implications of Energy Conditions on Standard Static Space-times}
\author{Fernando Dobarro}
%
%%%%\address[F. Dobarro]{Mailing address: ICTP, Mathematics Section,
%%%%Strada Costiera 11, I-34014 Trieste, Italy}
%%%%\email{dobarro@dmi.units.it}
%
\address[F. Dobarro]{Dipartimento di Matematica e Informatica,
Universit\`{a} degli Studi di Trieste, Via Valerio 12/B, I-34127
Trieste, Italy} \email{dobarro@dmi.units.it}
\author{B\"{u}lent \"{U}nal}
\address[B. \"{U}nal]{Department of Mathematics, Bilkent University,
         Bilkent, 06800 Ankara, Turkey}
\email{bulentunal@mail.com} \keywords{Warped products, standard
static space-times, energy conditions, sub-harmonic functions,
scalar curvature map, conformal hyperbolicity, conjugate points.}
%
%
% 35Q75 PDE in relativity
% 53C21 Methods of Riemannian geometry, including PDE methods, curvature restrictions
% 53C25 Special Riemannian manifolds (Einstein, ...)
% 53C50 Manifolds with indefinite metrics
% 53C80 Applications to Physics
% 83E15 Kaluza-Klein theory
% 83E30 Strings theory
%
\subjclass{53C21, 53C50, 53C80}
%\date{April , 2008}
\date{\today}
%\thanks{}

\begin{abstract}
%In this paper, we state a family of sufficient or necessary
%conditions (or both) for a set of energy (convergence) conditions
%on standard static space-times. We study implications of these
%energy (convergence) conditions to standard static space-times in
%terms of sub-harmonicity of the warping function and conformal
%hyperbolicity as well as conjugate points and the time-like
%diameter of this class of space-times.
%
In the framework of standard static space times, we state a family
of sufficient or necessary conditions for a set of physically
reasonable energy and convergence conditions in relativity and
related theories. We concentrate our study on questions about the
sub-harmonicity of the warping function, the scalar curvature map,
conformal hyperbolicity, conjugate points and the time-like
diameter of this class of space-times.
\end{abstract}

\maketitle

\tableofcontents

\newpage

\renewcommand{\thepage}{\arabic{page}}
\setcounter{page}{1}

\section{Introduction}

This paper deals with the study of energy conditions on standard
static space-times. Our first objective is to obtain a family of
necessary and/or sufficient conditions for a set of energy
conditions on standard static space-times. The second aim is to
apply the latter to a group of questions about conformal
hyperbolicity (in the sense of M. J. Markowitz) and the existence
of conjugate points in the same framework. Especially, we also pay
attention to a set of partial differential operators involved in
these discussions
\footnote{We would like to inform the reader that some of the
results provided in this article were previously announced in
\cite{Dobarro-Unal2008}.}.

Throughout our study, the warped product of manifolds is the
underlying central concept. The warped product of
pseudo-Riemannian manifolds were introduced in general relativity
as a method to find general solutions to Einstein's field
equations. Two important examples include generalized
Robertson-Walker space-times and standard static space-times. The
latter class can be regarded as a generalization of the Einstein
static universe.

\noindent We recall that a warped product can be defined as
follows \cite{BEE,ON}. Let $(B,g_B)$ and $(F,g_F)$ be
pseudo-Riemannian manifolds and also let $b \colon B \to
(0,\infty)$ be a smooth function. Then the (singly) \textit{warped
product}, $B \times {}_b F$ is the product manifold $B \times F$
furnished with the metric tensor $g=g_B \oplus b^{2}g_F$ defined
by
\begin{equation*}
    g=\pi^{\ast}(g_B) \oplus (b \circ \pi)^2 \sigma^{\ast}(g_F),
\end{equation*}
where $\pi \colon B \times F \to B$ and $\sigma \colon B \times F
\to F$ are the usual projection maps and ${}^\ast$~denotes the
pull-back operator on tensors.

A \textit{standard static space-time} (also called
\textit{globally static}, see \cite{KobayashiO-Obata1981}) is a
Lorentzian warped product where the warping function is defined on
a Riemannian manifold (called the natural space or Riemannian
part) and acting on the negative definite metric on an open
interval of real numbers. More precisely, a standard static
space-time, denoted by $I_f \times F$, is a Lorentzian warped
product furnished with the metric $g=-f^2{\rm d}t^2 \oplus g_F,$
where $(F,g_F)$ is a Riemannian manifold, $f \colon F \to
(0,\infty)$ is smooth and $I=(t_1,t_2)$ with $-\infty \leq t_1 <
t_2 \leq \infty$. In \cite{ON}, it was shown that any static
space-time
\footnote{An $n-$dimensional space-time $(M,g)$ is called
\textit{static} if there exists a nowhere vanishing time-like
Killing vector field $X$ on $M$ such that the distribution of each
$(n-1)-$plane orthogonal to $X$ is integrable (see
\cite[Subsection 3.7]{BEE} and also the general relativity texts
\cite{Hall2004,Hughston-Tod1990,Sachs-Wu1977}).}
is \textit{locally} isometric to a standard static space-time.

Standard static space-times have been previously studied by many
authors. O. Kobayashi and M. Obata \cite{KobayashiO-Obata1980}
stated the geodesic equation for this kind of space-times. The
causal structure and geodesic completeness were considered in
\cite{Allison1988-b}, where sufficient conditions on the warping
function for causal geodesic completeness of the standard static
space-time were obtained (see also \cite{RASM}).
%
%In \cite{Allison1988-a}, conditions are found which guarantee that standard
%static space-times either satisfy or else fail to satisfy certain
%curvature conditions from general relativity.
%
The existence of geodesics in standard static space-times has been
studied by several authors. In \cite{mS2}, S\'{a}nchez gives a
good overview of geodesic connectedness in semi-Riemannian
manifolds, including a discussion for standard static space-times
(see also \cite{MS05,MS06}). The geodesic structure of standard
static space-times has been studied in \cite{GES} and conditions
are found which imply nonreturning and pseudoconvex geodesic
systems. As a consequence, it is shown that if the complete
Riemannian factor $F$ satisfies the nonreturning property and has
a pseudoconvex geodesic system and the warping function $f \colon
F \to (0,\infty)$ is bounded from above, then the standard static
spacetime $(a,b) _f\times F$ is geodesically connected. In
\cite{Dobarro-Unal2004}, some conditions for the Riemannian factor
and the warping function of a standard static space-time are
obtained in order to guarantee that no nontrivial warping function
on the Riemannian factor can make the standard static space-time
Einstein.

In general relativity and related theories, the energy conditions
are a set of physically reasonable imposed constraints to the
underlying space-time (for a deeper discussions about the several
energy conditions see
\cite{Carter2002,Frolov-Novikov1997,HE,ON,Visser1996,Wald1984}
among many others). In \cite{Allison1985,Allison1988-a} the author
investigates conditions on the warping function which guarantee
that a standard static space-times either satisfy or else fail to
satisfy certain energy or convergence conditions. Some part of our
results in the following sections are narrowly related to those of
D. Allison.

Now, we briefly mention some properties of  the Lorentzian
pseudo-distance $d_M$ on an $n(\geq 3)-$dimensional Lorenztian
manifold $(M,g)$ defined by M. J. Markowitz in \cite{MM1, MM2},
where the author follows the procedure developed by S. Kobayashi
in \cite{KobayashiS1970} (see also
\cite{KobayashiS1998,KobayashiS2005}). The Lorentzian
pseudo-distance $d_M$ depends only on the conformal class, that
is, it remains the same for all conformal metrics to $g$. It is
known that for strongly causal space-times, the causal structure
is equivalent to the conformal structure and hence this causes a
link between $d_M$ and the causal structure of $(M,g)$. At this
point, we want to emphasize that the classical result about the
equivalency of the causal structure and the conformal structure is
proven to be true even for distinguishing space-times (see
\cite{GPSe2}) but the same authors points out that the causal
structure and the conformal structure should not be identified
(see \cite{GPSe1}). A Lorentzian manifold $(M,g)$ is called
conformally hyperbolic if the Lorentzian pseudo-distance $d_M$
satisfies all the conditions of an arbitrary distance function. In
%\cite[\textsc{Theorem 5.8}]{MM1}
%\cite{MM1}
\cite[\textsc{Theorem 5.8}]{MM1}, it is proven that an $n(\geq
3)$-dimensional Lorentzian manifold $(M,g)$ is conformally
hyperbolic if it satisfies the null convergence condition and the
null generic condition. Hence, conformal hyperbolicity becomes a
natural property for physically realistic space-times. Moreover, a
conformally hyperbolic Lorentzian manifold is causally incomplete
in sense of Markowitz (see \cite[Proposition 3.3]{MM1}). Here, we
would like to have the attention of the reader to the difference
between the usual causal completeness (see \cite{BEE}) and causal
completeness \textit{in the sense of Markowitz} (see
\cite{MM1,MM2}) of a space-time. According to the former, every
causal geodesic must be defined on the set of all real numbers but
according to the latter, every \textit{null} geodesic must be
extended to arbitrary values of every \textit{projective
parameter}. On the other hand, a null geodesically complete
Lorentzian manifold satisfying the reverse null convergence
condition, namely ${\rm Ric}(\mathrm{v},\mathrm{v}) \leq 0$ for
any null tangent vector $\mathrm{v}$, has the trivial
pseudo-distance, i.e., $d_M \equiv 0$ (see \cite[Theorem
5.1]{MM1}).

\noindent By applying the results in \cite{MM1}, Markowitz studied
the conformal hyperbolicity of generalized Robertson-Walker
space-times in \cite{MM2}. He also computed explicitly the
Lorentzian pseudo-distance $d_M$ on the Einstein-de Sitter
space-time by conformally imbedding this space-time into Minkowski
space.

%\noindent At this point, we will provide some notation. Let $(F,g_F)$
%be a Riemannian manifold. We will denote $C^{\infty}_{>0}(F)$ the set
%of strictly positive functions in $C^{\infty}(F)$. Besides,
%$\mathfrak X(F)$ will denote the $C^{\infty}(F)-$module of vector
%fields on $F$.

\bigskip

%After providing some of the previous studies related to our results
%given in the following sections, we now provide an outline of the
%paper.

After the previous  brief description of the generic scenario for
our study, we now provide an outline of the paper.

In {\it Section \ref{sec:preliminaries}}, we recall some
definitions and suitable expressions of Ricci and scalar
curvatures for a standard static space-time.

\noindent We also recall some topics in the study of the scalar
curvature map $\tau: g_F \mapsto \tau_{g_F}$, where $g_F$ is a
Riemannian metric on a given manifold $F$ and $\tau_{g_F}$ denotes
the associated scalar curvature of $(F,g_F)$. In the analysis of
$\tau $ it is usually useful to consider the linearization of this
map, i.e., $\mathscr{L}_{g_F}h = d \tau(g_F+th)/dt|_{t=0}$ . It
turns out that $\mathscr{L}_{g_F}$ is elliptic and the study of
the kernel of its formal adjoint $\mathscr{L}_{g_F}^\ast f
%:= -f \, {\rm Ric}_{g_F}- \Delta_{g_F} f
%\,\, g_F +{\rm H}^f _{g_F}
$ plays a central role to analyze the surjectivity of
$\mathscr{L}_{g_F}.$ In the 70's Bourguignon and Fischer-Marsden
showed that: if $(F, g_F)$ is a compact manifold, a necessary
condition for a nontrivial kernel of $\mathscr{L}_{g_F}^\ast$ is
that $\tau_{g_F}$ be a nonnegative constant.

\noindent On the other hand, for a given smooth function $f \colon
F \to (0,\infty)$, we introduce the $2-$covariant tensor
\begin{equation*}\label{}
Q_{g_F}^f:=\Delta_{g_F} f \, g_F- H_{g_F}^f,
\end{equation*}
where $(F,g_F)$ is a Riemannian manifold and $\Delta_{g_F}
(\cdot)= g_F^{ij} \nabla^{g_F}_i\nabla^{g_F}_j (\cdot)$ and
$H_{g_F}$ are the associated Laplace-Beltrami operator and Hessian
tensor, respectively. We consider also the Ricci tensor of
$(F,g_F)$, denoted by $\Ric_{g_F}$.

\noindent The tensors $\mathscr{L}_{g_F}^\ast f$, $Q_{g_F}^f$,
$\Ric_{g_F}$ and their associated quadratic forms play a central
role in the next two sections.

In \textit{Section \ref{sec:ECandI}}, after recalling a set of
energy and convergence conditions relevant in general relativity,
we obtain a family of necessary and/or sufficient conditions for
them
%energy and convergence conditions
on a standard static space-time. In \textit{Subsection
\ref{subsec:energy conditions}} we introduce the energy and
convergence conditions and make some generic comments about them.
In the first part of \textit{Subsection \ref{subsec:liouville}},
we state a family of results that show up a connection among the
strong energy condition and a family of Liouville type results for
subharmonic functions on the Riemannian part of a standard static
space-time (see \textit{Theorem \ref{thm:te-c1}} and its
\textit{corollaries}). In the second part, the principal results
are \textit{Theorems \ref{thm:ec-t}, \ref{thm:ec-w},
\ref{thm:ec-w2}, \ref{thm:skid}} and their corollaries, notice
particularly \textit{Corollary \ref{cor:wec-1}}.
All these results are based on suitable hypothesis for the
definiteness of the quadratic forms associated to the above
mentioned tensors $\mathscr{L}_{g_F}^\ast f$, $Q_{g_F}^f$ and
$\Ric_{g_F}$.

\noindent At the end of this subsection, in \textit{Theorem
\ref{thm:einstein-corvino}}, we give a \2partial" extension of the
Bourguignon/Fischer-Marsden result mentioned above to the case
where the involved manifold is complete but noncompact and with
nonnegative Ricci curvature.

In \textit{Section \ref{sec:ConseqExamp}}, we combine the results
in \textit{Section \ref{sec:ECandI}} (particularly,
\textit{Theorem \ref{thm:ec-t}}) with the results of Markowitz.

\noindent In {\it Subsection \ref{subsec:ConfHyper}}, we obtain
sufficient conditions for a standard static space-time to be
conformally hyperbolic. In brief, we obtain that if the quadratic
forms associated to the Ricci tensor on the natural space and
$Q_{g_F}^f$ are positive semi-definite and positive definite,
respectively, then a standard static space-time of the form
$I_f\times F$ is conformally hyperbolic (see {\it Theorem
\ref{main-3}}).

\noindent In {\it Subsection \ref{subsec:ConjugatePoints}}, we
establish sufficient conditions in order to guarantee that any
causal geodesic on a standard static space-time has a pair of
conjugate points. By a stronger set of hypothesis, we also obtain
an estimate from above for the time-like diameter of the standard
static space-time (see {\it Corollaries \ref{cor-1} and
\ref{cor-3}}).

\noindent Finally, in {\it Subsection \ref{subsec:special
examples}}, we show some results connecting the tensor
$Q_{g_F}^f$, conformal hyperbolicity, concircular scalar fields
and Hessian manifolds. More precisely, we give sufficient
conditions for a standard static space-time to be conformally
hyperbolic where the Riemannian part admits a concircular scalar
field or is a global Hessian manifold.

\section{Preliminaries}\label{sec:preliminaries}

Throughout the article $I$ will denote an open real interval of
the form $I=(t_1,t_2)$ where $-\infty \leq t_1 < t_2 \leq \infty.$
Moreover, $(F,g_F)$ will denote a connected Riemannian manifold
without boundary with $\dim F=s$. Finally, on an arbitrary
differentiable manifold $N$, $C^{\infty}_{>0}(N)$ denotes the set
of all strictly positive $C^{\infty}$ functions defined on $N$,
$TN = \bigcup_{p \in N} T_p N$ denotes the tangent bundle of $N$
and $\mathfrak X(N)$ will denote the $C^{\infty}(N)-$module of
smooth vector fields on $N.$
\footnote{Notation: In the present study, manifolds are denoted by
$B,F,M,N$; points by $p,q,x,y$; vectors by $\mathrm{v},\mathrm{w},
\mathrm{x}, \mathrm{y}$ and also vector fields by $V,W,X,Y$.}

\begin{dfn} \label{def:sss-t} Let $f \in C^{\infty}_{>0}(F)$.
The $n(=1+s)-$dimensional product manifold $I \times F$ furnished
with the metric tensor $g=-f^2{\rm d}t^2 \oplus g_F$ is called a
\textit{standard static space-time} (also usually called
\textit{globally static}, see \cite{KobayashiO-Obata1981}) and is
denoted by $ I _f\times F$.
\end{dfn}

\begin{dfn} \label{def:RW} Let $b \in C^{\infty}_{>0}(I)$.
The $n(=1+s)-$dimensional product manifold $I \times F$ endowed
with the metric tensor $g=-{\rm d}t^2 \oplus b^2 g_F$ is called a
\textit{generalized Robertson-Walker space-time} and following the
warped product notation is denoted by $I \times_b F$.
\end{dfn}

%We observe that we can identify a vector field on the base
%of a warped product with its lift to the product by the
%corresponding projection (see \cite{ON}).

On a warped product of the form $B \times_f F$, we will denote the
set of lifts to the product by the corresponding projection of the
vector fields in $\mathfrak X (B)$ (respectively, $\mathfrak X
(F)$) by $\mathfrak L (B)$ (respectively, $\mathfrak L (F)$) (see
\cite{ON}). We will use the same symbol for a tensor field and its
lift.
\medskip

The following formula of the curvature Ricci tensor can be easily
obtained from \cite{BEE,Dobarro-Unal2004,CMP,ON}.

%%%%%%%%%%%%%%%%%%%%%%%%%%%%%%%%%%%%%%%%%%%%%%%%%%%%%%%%%%%%%%%%%%%%%%%
%\begin{prp} \label{prp:ricci-ssst} Let $I _f\times F$ be a standard static
%space-time furnished with the metric $g=-f^2 {\rm d}t^2 \oplus
%g_F.$ Suppose that $U \in \mathfrak X (I)$ and $V,W \in \mathfrak
%X (F)$. If ${\rm Ric}$ and ${\rm Ric}_F$ denote the Ricci tensors
%of $I _f\times F$ and $(F,g_F)$, respectively, then
%$${\rm \Ric} \left(U+V, U+W \right)={\rm Ric}_F(V,W)+ f \Delta_F f \,
%{\rm d}t^2(U,U) - \frac{1}{f}{\rm H}^f_F(V,W),$$ where ${\rm
%H}^f_F$ is the Hessian form of $f$ on $(F,g_F).$
%\end{prp}
%%%%%%%%%%%%%%%%%%%%%%%%%%%%%%%%%%%%%%%%%%%%%%%%%%%%%%%%%%%%%%%%%%%%%%%

\begin{prp} \label{prp:ricci-ssst} Let $I _f\times F$ be a standard static
space-time furnished with the metric $g=-f^2 {\rm d}t^2 \oplus
g_F.$ Suppose that $U_1,U_2 \in \mathfrak L (I)$ and $V_1,V_2 \in
\mathfrak L (F)$. If ${\rm Ric}$ and ${\rm Ric}_{g_F}$ denote the
Ricci tensors of $I _f\times F$ and $(F,g_F)$, respectively, then
%
%%%%
%%%\begin{eqnarray*}%\label{eq:ricci-ssst}
%%%   &{\rm \Ric} \left(U_1+V_1, U_2+V_2 \right)=\\
%%%   &\displaystyle {\rm Ric}_F(V_1,V_2)+ f \Delta_F f \, {\rm d}t^2(U_1,U_2) - \frac{1}{f}{\rm H}^f_F(V_1,V_2) =\\
%%%   &\displaystyle {\rm Ric}_F(V_1,V_2)+ \frac{1}{f} \Delta_F f \, g_F (V_1,V_2) - \frac{1}{f}{\rm H}^f_F(V_1,V_2)
%%%   -\frac{1}{f} \Delta_F f \, g(U_1+V_1, U_2+V_2) =\\
%%%   &\displaystyle -\frac{1}{f} \mathscr{L}_{g_F}^\ast f (V_1,V_2)
%%%   -\frac{1}{f} \Delta_F f \, g(U_1+V_1, U_2+V_2)
%%%\end{eqnarray*}
%%%%
%
\begin{equation}\label{eq:ricci-ssst}
    \begin{split}
   {\rm \Ric} &(U_1+V_1, U_2+V_2 )\\
   &=\displaystyle {\rm Ric}_{g_F}(V_1,V_2)
   + f \Delta_{g_F} f \, {\rm d}t^2(U_1,U_2) - \frac{1}{f}{\rm H}_{g_F}^f(V_1,V_2) \\
   %&=\displaystyle {\rm Ric}_F(V_1,V_2)+ \frac{1}{f} \Delta_F f \, g_F (V_1,V_2) - \frac{1}{f}{\rm H}^f_F(V_1,V_2)
   %-\frac{1}{f} \Delta_F f \, g(U_1+V_1, U_2+V_2) \\
   &=\displaystyle -\frac{1}{f} \mathscr{L}_{g_F}^\ast f (V_1,V_2)
   -\frac{1}{f} \Delta_{g_F} f \, g(U_1+V_1, U_2+V_2)
    \end{split}
\end{equation}
%
%
%%%%%$\displaystyle \mathscr{L}_{g_F}^\ast f := -f \, {\rm Ric}_F-
%%%%%\Delta_F f \, g_F +{\rm H}^f_F$
%
%%%%%\begin{equation}\label{eq:ricci-ssst}
%%%%%\begin{split}
%%%%%{\rm \Ric} \left(U_1+V_1, U_2+V_2 \right)=&{\rm Ric}_F(V_1,V_2)+ f
%%%%%\Delta_F f \, {\rm d}t^2(U_1,U_2)\\  &- \frac{1}{f}{\rm
%%%%%H}^f_F(V_1,V_2),
%%%%%\end{split}
%%%%%\end{equation}
where ${\rm H}_{g_F}^f$ is the Hessian tensor of $f$ on $(F,g_F)$
and $\mathscr{L}_{g_F}^\ast f$ is the $2-covariant$ tensor given by
\begin{equation}\label{eq:adjoint-lsc}
    \displaystyle \mathscr{L}_{g_F}^\ast f := -f \, {\rm Ric}_{g_F}-
\Delta_{g_F} f \, g_F +{\rm H}_{g_F}^f
\end{equation}
on $(F,g_F)$.
\end{prp}

In \cite{Dobarro-Unal2004}, we studied the problem of constancy of
the scalar curvature on a standard static space-time, in
particular we obtained the following formula.

\begin{prp} \label{prp:sc-ssst} Let $I _f\times F$ be a standard static
space-time furnished with the metric $g=-f^2 {\rm d}t^2 \oplus
g_F.$ If $\tau$ and $\tau_{g_F}$ denote the scalar curvatures of
of $I _f\times F$ and $(F,g_F)$, respectively, then
$$ \tau = \tau_{g_F} -2 \frac{1}{f}\Delta_{g_F} f.$$
As a consequence $f$ is sub-harmonic (i.e., $\Delta_{g_F} f \geq 0$)
if and only if $\tau < \tau_{g_F}$ (since $f \in C^{\infty}_{>0}(F)$).
\end{prp}

Now, we will introduce a special tensor frequently used to
establish energy conditions on a standard static space-time (see
for instance \cite{Allison1985,Allison1988-a}). This tensor will
be play a central role in the next sections.

\begin{notation} \label{notation: quadratic form}
Let $(F,g_F)$ be a Riemannian manifold and $f \in
C^{\infty}_{>0}(F)$. From now on, $Q_{g_F}^f$ denotes the
$2-$covariant tensor given by
\begin{equation}\label{eq: quadratic form 1}
    Q_{g_F}^f:=\Delta_{g_F} f \, g_F- H_{g_F}^f.
\end{equation}
If it is necessary, we will emphasize the evaluation points $p \in
F$ by writing $Q_{g_F|p}^f$. We will apply the latter convention
for other tensors too.

Furthermore, we will denote its associated quadratic form by
$\mathcal{Q}_{g_F}^f$ for any point $p \in F$. More precisely, for
any $\mathrm{v} \in {\rm T}_p F$
\begin{equation}\label{eq:quadratic form 2}
    \mathcal{Q}_{g_F|p}^f(\mathrm{v}):=Q_{g_F|p}^f(\mathrm{v},\mathrm{v}),
\end{equation}
where ${\rm T}_p F$ is the tangent space to $F$ at $p$.

Analogously, we will denote the associated quadratic form to the
Ricci tensor $\Ric$ by $\mathcal{R}ic$. So, for any \textit{unit}
tangent vector $\mathrm{v} \in {\rm T}_p F$
\begin{equation}\label{eq:Ricci quadratic form 2}
    \mathcal{R}ic_{g_F|p}(\mathrm{v}):=\Ric_{g_F|p}(\mathrm{v},\mathrm{v})
\end{equation}
is so called the Ricci curvature in the direction of $\mathrm{v}$
(see \cite{Aubin98,Berard1986,BerardBergery1980})).

Besides, we will denote the associated quadratic form to the $2$
covariant tensor $\mathscr{L}_{g_F}^\ast f$ by
$\mathcal{L}_{g_F}^\ast f $.
\end{notation}

\begin{rem}\label{rem:quadratic form1}
\footnote{ We observe that this result is valid on any
pseudo-Riemannian manifold if one takes the signature of the
metric in the definition of the divergence into account.}
It is easy to prove that the 2-covariant tensor $Q_{g_F}^f$ is
\textit{divergence-free}. Indeed, by definition of the involved
differential operators, the identity stated as $\diver_{g_F}
(\phi g_F) = \mathrm{d} \phi$ for all $\phi \in C^{\infty}(F)$
and commuting derivatives (see \cite[p. 85-87]{ON}),
\begin{equation}\label{eq:Qfdiv-free}
    \begin{split}
    (\diver_{g_F} Q_{g_F}^f)_j &= (\diver_{g_F} ((\Delta_{g_F} f) \, g_F) - \diver_{g_F} H_{g_F}^f)_j\\
                 &= (\mathrm{d} (\Delta_{g_F} f) - \diver_{g_F} H_{g_F}^f)_j\\
                 &= \nabla_j \nabla^k \nabla_k f  - \nabla^k \nabla_j
                 \nabla_k f\\
                 &= \nabla_j \nabla^k \nabla_k f  - \nabla_j \nabla^k
                 \nabla_k f\\
                 &=0,
    \end{split}
\end{equation}
for any $j$, where $\nabla=\nabla^{g_F}$ is the Levi-Civita
connection associated to $(F,g_F)$.
%
%We observe also that this result holds on semi-Riemannian
%manifolds taking into account the signature of the metric in the
%definition of the divergence.
\end{rem}

\begin{rem} Recall that if $f$ is convex (or concave, respectively),
i.e., ${\rm H}_{g_F}^f(\mathrm{v},\mathrm{v}) \geq 0$ for any
$\mathrm{v}\in TF$ (or ${\rm H}_{g_F}^f(\mathrm{v},\mathrm{v})
\leq 0$ for any $\mathrm{v} \in TF$, respectively) then
$\Delta_{g_F} f \geq 0$ (or $\Delta_{g_F} f \leq 0,$
respectively).
\end{rem}

\begin{rem}\label{rem:quadratic form}
Let $Q_{g_F}^f$ be defined as above with $\dim F = s \geq 2$.
\begin{itemize}
    \item[{\bf (i)}] If the symmetric bilinear form $Q_{g_F|p}^f$
    is positive (respectively, negative) semi-definite for a
    point $p \in F,$ then $(\Delta_{g_F} f)_{|p} \ge 0$ (respectively,
    $\le 0$). It is sufficient to observe that $\trace_{g_F}
    Q_{g_F}^f = (s-1)\Delta_{g_F} f$ and the $g_F-$trace preserves the
    sign.
    \item[{\bf (ii)}] It is clear that \textbf{(i)} holds while
    replacing semi-definiteness by definiteness and accordingly,
    $\ge$ by $>$.
    \item[{\bf (iii)}] Assume now that $F$ is compact and $f
    \colon F \to \mathbb R$ is smooth. The so called Bochner's
    Lemma (see \cite[p. 39]{Yano1970}) says: if $\Delta_{g_F} f \geq 0$,
    then $f$ is constant. Thus, it just follows from
    \textbf{(ii)} that $f$ is constant if $Q_{g_F}^f$ is positive
    semi-definite. \\
    Furthermore, if $Q_{g_F}^f$ is negative semi-definite, then
    $Q_{g_F}^{-f}=-Q_{g_F}^f$ is positive semi-definite, so $-f$
    (and obviously $f$) is constant.
\end{itemize}
\end{rem}

\begin{rem}\label{rem:adjointsc}
Notice that applying \eqref{eq: quadratic form 1}, the
$(0,2)-$tensor $\mathscr{L}_{g_F}^\ast f$ defined in
\eqref{eq:adjoint-lsc} takes the form
\begin{equation}\label{eq:adjoint-lsc-Q}
    \displaystyle \mathscr{L}_{g_F}^\ast f := -f \, {\rm Ric}_{g_F}- {\rm Q}_{g_F}^f.
\end{equation}
This tensor is strongly associated to the scalar curvature map
between Banach manifolds that apply to each metric on a smooth
manifold (eventually, to a sub-domain) the corresponding scalar
curvature. Indeed, if $g_F$ is a suitable Riemannian metric on
such a manifold and $f$ is a sufficiently regular function on the
manifold, $\mathscr{L}_{g_F}^\ast f$ is the $L^2-$formal adjoint
of the linearized scalar curvature operator. There is a large
literature about these operators, for instance
\cite{Lichnerowicz1961,Fischer-Marsden1974,Marsden1974,Fischer-Marsden1975-a,
Fischer-Marsden1975-b,Bourguignon1975,
KobayashiO-Obata1980,KobayashiO-Obata1981,KobayashiO1982,Lafontaine1983,
Mazzeo-Pollack-Uhlenbeck1994,
Corvino2000,Hwang2000,Lafontaine-Rozoy2000,Hwang2003,Corvino-Schoen2006,Lafontaine2008}
and the many references therein.

The kernel of the operator \eqref{eq:adjoint-lsc-Q} will play an
important role in our discussion about the dominant energy
condition in the following section.
\end{rem}

\section{Energy Conditions and Implications}\label{sec:ECandI}

There are several natural energy conditions considered in general
relativity and cosmology questions (see
\cite{BEE,Frolov-Novikov1997,HE,ON,Visser1996,Wald1984} among
others).
They are related for instance with: space-time singularities (see
\cite{Tipler1}), existence of conjugate points (see
\cite{BordeA}), splitting theorems and existence of time-like
lines (see \cite{Esh-Galo}), Lorentzian wormholes (see
\cite{Visser1996}), higher dimensional black-holes (see
\cite{Cai-Galloway2001,Emparan-Reall2008,Galloway-Schoen2006}),
spacelike foliations (see \cite{Nardmann2007}), dS/CFT and
space-time topology (see \cite{Andersson-Galloway2002}) among many
other topics.

In the first subsection we will introduce some energy conditions
on generic space-times, while in the second subsection we will
state some of these for a standard static space-time of the form
$M=I _f\times F$ with the metric $g=-f^2{\rm d}t^2 \oplus g_F$ in
terms of the tensors $\mathscr{L}_{g_F}^\ast f$ and $Q_{g_F}^f$
and analyze possible consequences of these conditions.

\subsection{Definitions, Generalities and more}
\label{subsec:energy conditions} In this subsection, let us
consider a space-time $(N,g_N)$ of dimension $n \geq 3$.

Assuming the conventions in \cite{Allison1988-a, Frankel1979,
Lee1975}; the space-time $(N,g_N)$ is said to satisfy the
\emph{strong energy condition} (respectively, \emph{reverse strong
energy condition}), briefly SEC (respectively, RSEC), if ${\rm
Ric}(\mathrm{x},\mathrm{x}) \geq 0$ (respectively, ${\rm
Ric}(\mathrm{x},\mathrm{x}) \leq 0$) for all causal tangent
vectors $\mathrm{x}$.

Furthermore $(N,g_N)$ is said to satisfy the \emph{time-like}
(respectively \emph{null, space-like}) \emph{convergence
condition}, briefly TCC (respectively NCC, SCC), if ${\rm
Ric}(\mathrm{x},\mathrm{x}) \geq 0$ for all time-like
(respectively null, space-like) tangent vectors $\mathrm{x}$. The
corresponding \emph{reverses} (i.e., ${\rm
Ric}(\mathrm{x},\mathrm{x}) \leq 0$) will be denoted as above by
RTCC, RNCC and RSCC, respectively.

Notice that the SEC implies the NCC. Furthermore the TCC is
equivalent to the SEC, by continuity. The actual difference
between TCC and SEC follows from the fact that while TCC is just a
geometric condition imposed on the Ricci tensor, SEC is a
condition on the stress-energy tensor. They can be considered
equivalent due to the Einstein equation (see \eqref{eq: Einstein
eq} and the discussion about the SEC definition adopted by us and
those in the sense of Hawking and Ellis).
\footnote{One can refer \cite[p. 434]{BEE} for further interesting
comments about the different definitions and conventions about these
energy and convergence conditions.}

Moreover, a space-time is said to satisfy the \emph{weak energy
condition}, briefly WEC, if ${\rm T}(\mathrm{x},\mathrm{x}) \geq
0$ for all time-like vectors, where ${\rm T}$ is the
energy-momentum tensor, which is determined by physical
considerations.

In this article, when we consider the energy-momentum tensor, we
assume that Einstein's equation holds (see \cite{HE,ON}). More
explicitly,
\begin{equation}\label{eq: Einstein eq}
\Ric -\frac{1}{2} \tau g_N = 8 \pi {\rm T}.
\end{equation}
Notice that in particular, \eqref{eq: Einstein eq}
gives the explicit form of the energy-momentum tensor
$\mathrm{T}$.

The WEC has many applications in general relativity theory such as
nonexistence of closed time-like curve (see \cite{ChoPark}) and
the problem of causality violation (\cite{OriSoen}). But its
fundamental usage still lies in Penrose's Singularity theorem (see
\cite{Pen1}).

Notice that the notion of SEC considered in our study is not the
same one as stated in \cite[p. 95]{HE} (see also \cite[p. 434]{BEE},
\cite{Beem-Ehrlich1978}). Indeed for Hawking and Ellis, $(N,g_N)$
verifies the strong energy condition (briefly HE-SEC) if and only
if
\begin{equation}\label{eq:HE-SEC}
    \mathrm{T}(\mathrm{x},\mathrm{x}) \geq \frac{1}{2}\trace \mathrm{T}\,g_N(\mathrm{x},\mathrm{x}),
\end{equation}
for all time-like tangent vectors $\mathrm{x}$.

On the other hand, by \eqref{eq: Einstein eq} is
\begin{equation*}\label{eq:}
\frac{2-n}{2}\tau = 8 \pi \trace\mathrm{T}.
\end{equation*}
Thus
\begin{equation}\label{eq:Ricci-T}
    \Ric = 8 \pi \left[\mathrm{T} - \frac{1}{n-2}\trace \mathrm{T}\,
    g_N\right].
%
%    =8 \pi \left[\mathrm{T} + \frac{1}{8 \pi}\frac{1}{2}\tau g_N\right]
%
\end{equation}
So the HE-SEC is equivalent to the SEC, if and only if $n=4$.
Furthermore,
\begin{itemize}
    \item if $n > 4$ and $\tau \geq 0$, HE-SEC implies SEC
    \item if $n > 4$ and $\tau \leq 0$, SEC implies HE-SEC
    \item if $n \geq 3$ and $\tau = 0$, then SEC, HE-SEC and WEC are equivalent.
\end{itemize}

At this point we immediately observe that neither SEC nor HE-SEC implies
WEC (see \cite[p. 117]{Visser1996} for dimension $4$).
However,\begin{itemize}
    \item SEC + `` $\tau \geq 0$ " implies WEC
    \item WEC + `` $\tau \leq 0$ " implies SEC
%    \item
\end{itemize}
Consequently,
\begin{itemize}
    \item if $\tau = 0$, SEC is equivalent to WEC.
\end{itemize}
Other important energy condition is the so called \textit{dominant
energy condition} (briefly DEC), namely: WEC + `` $\mathrm{T}^{ij}
\mathrm{x}_{j}$
%[$-\sharp(T(x,\cdot))$ = $T^{ab}x_a$] [$\sharp:T*N \longrightarrow TN$]
is causal for all time-like vectors $\mathrm{x}$ ". It is clear
that DEC implies WEC. This condition is extremely relevant, for
instance in the recent studies of higher dimensional black holes
\cite{Cai-Galloway2001,Emparan-Reall2008,Galloway-Schoen2006} (see
also \cite{Carter2002}).

\noindent An equivalent formulation of the DEC is (see
\cite{Galloway-Schoen2006,Gibbons2003,Andersson-Galloway2002}):
%
%The dominant energy condition provided the energy-momentum tensor
%$T$ satisfies
%
``$\,\mathrm{T} (\mathrm{x}, \mathrm{y} )
%= \mathrm{T}_{\mu \nu} x^\mu y^\nu
\geq 0$ for all future pointing causal vectors $\mathrm{x}$ and
$\mathrm{y}\,$".

%$T_{\mu \nu} t^\mu s^\nu \geq 0$, for all pairs of future directed
%timelike vectors $t^\mu$ and $s^\nu$.

\bigskip

\subsection{Some implications on standard static space-times}
\label{subsec:liouville}

Let $M$$=I _f\times F$ be a standard static space-time with the
metric $g=-f^2{\rm d}t^2 \oplus g_F$. Thus,
\begin{equation}\label{eq:timelike partialt}
    g(\partial_t,\partial_t)= -f^2 <0
\end{equation}
i.e., $\partial_t$ is time-like and by Proposition
\ref{prp:ricci-ssst},
\begin{equation}\label{eq:ricci partialt ssst}
\Ric (\partial_t,\partial_t)= f \Delta_{g_F} f.
\end{equation}
So we can easily state the following.
\begin{thm} \label{thm:te-c1} \cite{Allison1988-a} Let
$M=I _f\times F$ be a standard static space-time with the metric
$g=-f^2{\rm d}t^2 \oplus g_F.$ If the space-time $(M,g)$ satisfies
the SEC (or equivalently, the TCC), then $f$ is subharmonic, i.e., $\Delta_{g_F} f \geq 0$.
\end{thm}

In particular, we have the followings:

\begin{cor}\label{cor:te-c1-bis}Let $(F,g_F)$ be
admitting no nonconstant subharmonic positive function. Then, the
SEC is verified by no standard static space-time with \2natural
space $(F,g_F)$ and nonconstant warping function".
\end{cor}

%\begin{cor} \label{cor:te-c1-bisbis} Let $M=I _f\times F$ be a
%standard static space-time with the metric $g=-f^2{\rm d}t^2
%\oplus g_F.$ Suppose that there exists at least one point in $F$
%at which $\Delta_F f$ is negative. Then, the space-time $(M,g)$
%cannot satisfy the SEC.
%\end{cor}

\begin{cor} \label{cor:sec-compact} Let $M=I _f\times F$ be a
standard static space-time with the metric $g=-f^2{\rm d}t^2
\oplus g_F$ satisfying the SEC (or equivalently, the TCC).
If $(F,g_F)$ is compact, then $f$ is a positive constant.
\end{cor}

\begin{proof}
It is sufficient to apply \emph{Theorem \ref{thm:te-c1}} and the
Bochner Lemma (see \emph{Remark \ref{rem:quadratic form}}).
\end{proof}

Several sufficient conditions are known for nonexistence of
a nonconstant nonnegative subharmonic function on a complete
Riemannian manifold. Thus, again by \emph{Theorem
\ref{thm:te-c1}}, it is possible to obtain a family of results in
the cases where the natural space is a complete but noncompact
Riemannian manifold. Here, we will mention some of them based on
results of P. Li, R. Schoen and Shing Tung Yau (see
\cite{Li2000,Li2004} and references therein)
\footnote{From now on $L^p(F)$ denotes the usual norm space of
$p-$integrable functions on a manifold $F$, where $p$ is a real
number.}
.

\begin{cor} \label{cor:sec-complete} Let $M=I _f\times F$ be a
standard static space-time with the metric $g=-f^2{\rm d}t^2
\oplus g_F$ satisfying the SEC (or equivalently, the TCC),
where $(F,g_F)$ is a noncompact complete Riemannian manifold.
Suppose that $x_0 \in F$ is a fixed point and $\rho$ denotes the
distance function to $x_0$ in $(F,g_F)$. Then the warping function
$f$ is constant and $(F,g_F)$ has finite volume if at least one of
the following conditions is verified:
\begin{enumerate}
    \item $f \in L^p(F)$ for some $p > 1$.
    \item $f \in L^1(F)$ and there exist
        constants $C > 0$ and $\alpha > 0$ such that the Ricci
        curvature of $(F,g_F)$ satisfies $\Ric_{g_F|x} \ge -C (1 +
        \rho^2(x))(\log(1 + \rho^2(x) )^{-\alpha}$.
    \item $f \in L^p(F)$ for some $0<p<1$ and
        there exists a constant $\delta(s) >0$ depending only on $s$,
        such that the Ricci curvature of $(F,g_F)$ satisfies
        $\Ric_{g_F|x} \ge -\delta(s)\rho^{-2}(x)$ as $x \rightarrow \infty$.
    \item $f \in L^1(F)$ and there exists a
        constant $C > 0$ such that the Ricci curvature of $(F,g_F)$
        satisfies $\Ric_{g_F|x} \ge -C (1 + \rho^2(x))$.
\end{enumerate}
\end{cor}

\begin{proof} Notice that $f$ is positive and \emph{Theorem \ref{thm:te-c1}}
says that $f$ is sub-harmonic. Thus, we have the following
reasonings.
\begin{enumerate}
    \item It is sufficient to apply \cite[Theorem 12.1]{Li2000}
    (see also \cite{Yau1976}).
    \item It is sufficient to apply \cite[Theorem 12.2]{Li2000}
    (see also \cite{Li-Schoen1984}).
    \item It is sufficient to apply \cite[Theorem 12.3]{Li2000}
    (see also \cite{Li-Schoen1984}).
    \item It is sufficient to apply \cite[Theorem 12.4]{Li2000}
    (see also \cite{Li1984}).
    %\item It is sufficient to apply \cite[Theorem 12.5]{Li2000} (see also \cite{Li-Schoen1984}).
    %\item It is sufficient to apply \cite[Proposition 12.6]{Li2000} (see also \cite{Yau1976}).
\end{enumerate}
\end{proof}

In order to state the next corollary, we recall that a complete
Riemannian manifold is said to be a Cartan-Hadamard manifold if
it is simply connected and has nonpositive sectional curvature
(see \cite[p. 128]{Li2004}).

\begin{cor} \label{cor:sec-complete-2} Let $M=I _f\times F$ be a
standard static space-time with the metric $g=-f^2{\rm d}t^2
\oplus g_F$ satisfying the SEC (or equivalently, the TCC), where
$(F,g_F)$ is a complete noncompact Riemannian manifold. Assume that
$(F,g_F)$ satisfies one of the following conditions:
\begin{enumerate}
    \item $(F,g_F)$ is a Cartan-Hadamard manifold.
    \item $(F,g_F)$ has Ricci curvature bounded from below and the
    volume of every unit geodesic ball is uniformly bounded from below.
\end{enumerate}
Then, the warping function $f \in L^p(F)$ for some $0<p\leq 1$ only if
$f$ is constant.
\end{cor}

\begin{proof} As above, $f$ is positive, so \emph{Theorem \ref{thm:te-c1}}
says that $f$ is sub-harmonic. Hence, it is sufficient to apply
\cite[Theorem 12.5]{Li2000} (see also \cite{Li-Schoen1984}).
\end{proof}

\begin{cor} \label{cor:sec-complete-3} Let $M=I _f\times F$ be a standard
static space-time with the metric $g=-f^2{\rm d}t^2 \oplus g_F$
satisfying the SEC (or equivalently, the TCC), where $(F,g_F)$ is
a complete noncompact Riemannian manifold. Thus, if the $L^p-$norm
of the warping function $f$ satisfies
\begin{equation}\label{eq:Lp growth 1}
    \int_{B_x(r)} f^p = o(r^2),
\end{equation}
as $r \rightarrow \infty $ for some fixed point $x \in F$, then
the warping function $f$ is constant and
\begin{equation}\label{eq:Lp growth 2}
    \limsup_{r \rightarrow \infty } r^{-2} V_x (r)=0,
\end{equation}
where $B_x(r)$ is a geodesic ball in $(F,g_F)$ centered at $x$ of
radius $r$ and its volume is given by $V_x (r)$.
\end{cor}

\begin{proof} Similar to the proof of {\it Corollary
\ref{cor:sec-complete-2}}, the positivity of $f$ and
\emph{Theorem \ref{thm:te-c1}} imply that $f$ is
sub-harmonic. Thus, it is sufficient to apply
\cite[Proposition 12.6]{Li2000} (see also \cite{Yau1976}).
\end{proof}

\bigskip

Now, we state necessary conditions for a standard static
space-time to satisfy the NCC and other conditions of the
curvature Ricci tensor as in the assumptions of Theorems 5.1 and
5.8 in \cite{MM1} (for a detailed discussion about the energy
conditions see for instance \cite[Section 4.3]{HE}). These are
analog to the more accourate in \cite[Proposition
4.2]{Ehrlich-Sanchez2000} for the Generalized Robertson-Walker
space-time.

Simple consequences of \eqref{eq:ricci-ssst}, Notation
\ref{notation: quadratic form} and Remark \ref{rem:quadratic form}
are the following results (see also \cite[Theorems 3.3 and
3.6]{Allison1988-a}):

\begin{thm} \label{thm:ec-t} Let $M=I _f\times F$ be a standard
static space-time with the metric $g=-f^2{\rm d}t^2 \oplus g_F$,
where $s = \dim F \ge 2$.
\begin{enumerate}
\item $\mathcal{L}_{g_F}^\ast f $
is negative (respectively, positive) semi-definite if and only if
$M$ satisfies the NCC (respectively, RNCC).

\item If $\mathcal{R}ic_{g_F}$ and $\mathcal{Q}_{g_F}^f$ are positive
semi-definite, then $M$ satisfies the TCC and the NCC.

\item If $\mathcal{R}ic_{g_F}$ and $\mathcal{Q}_{g_F}^f$ are negative
semi-definite, then $\Ric(\mathrm{w},\mathrm{w}) \leq 0$ for any
causal
%(time-like or null)
vector $\mathrm{w} \in TM$, i.e., $M$ satisfies the RNCC and the
RTCC.

\item If $(F,g_F)$ is Ricci flat, then $\mathcal{Q}_{g_F}^f$ is positive
semi-definite if and only if $M$ satisfies the NCC.
\end{enumerate}
\end{thm}

\begin{rem} In order to obtain (1) in Theorem \ref{thm:ec-t}, it is useful to
note that for any $\mathrm{v} \in TF$, the vector field $r \,
f^{-1} \, \sqrt{g_F(\mathrm{v},\mathrm{v})}
\partial_t + \mathrm{v} \in TM$ is causal iff $|r|\geq 1$.
Moreover the vector field on $M$ is null iff $|r| = 1$.
\end{rem}

\bigskip

Now, we will deal with the energy conditions in terms of the
energy momentum-tensor (see \cite{Visser1996} for other results in
this direction). Let $M=I _f\times F$ be a standard static
space-time with the metric $g=-f^2{\rm d}t^2 \oplus g_F$. Recall
that we assume the validity of the Einstein's equation \eqref{eq:
Einstein eq}. On the other hand, tangent vectors $\mathrm{w} \in
T(I _f\times F)$ can be decomposed into $\mathrm{w} = \mathrm{u} +
\mathrm{v}$ with $\mathrm{u} \in TI$ and $\mathrm{v} \in TF$. It
is easy to obtain from \eqref{eq:ricci-ssst} and Proposition
\ref{prp:sc-ssst} that for any $U_1,U_2 \in \mathfrak X (I)$ and
$V_1,V_2 \in \mathfrak X (F)$
%%%
%%%\begin{equation}\label{eq:energy momentum ssst}
%%%\begin{split}
%%%8\pi \mathrm{T} \left(U_1+V_1, U_2+V_2 \right)=&{\rm
%%%Ric}_F(V_1,V_2) + \frac{1}{f} Q_F^f (V_1,V_2)
%%%\\  &
%%%- \frac{1}{2} \tau_F g(U_1+V_1,U_2+V_2).
%%%\end{split}
%%%\end{equation}
%%%
\begin{equation}\label{eq:energy momentum ssst}
\begin{split}
8\pi \mathrm{T} &\left(U_1+V_1, U_2+V_2 \right)=
\displaystyle -\frac{1}{f}\mathscr{L}_{g_F}^\ast f (V_1,V_2)-
\frac{1}{2} \tau_{g_F} g(U_1+V_1,U_2+V_2)\\
&={\rm Ric}_{g_F}(V_1,V_2) + \frac{1}{f} Q_{g_F}^f (V_1,V_2) -
\frac{1}{2} \tau_{g_F} g(U_1+V_1,U_2+V_2).
\end{split}
\end{equation}
%%%
In particular, for any $U \in \mathfrak X (I)$ and $V \in \mathfrak
X (F)$, we have
\begin{equation}\label{eq: em tensor}
\begin{split}
    8\pi \mathrm{T}&(U+V,U+V)=
    \displaystyle -\frac{1}{f}\mathcal{L}_{g_F}^\ast f (V,V)-
\frac{1}{2} \tau_{g_F} g(U+V,U+V)\\
    &=\mathcal{R}ic_{g_F}(V) + \frac{1}{f}
    \mathcal{Q}_{g_F}^f(V) - \frac{1}{2} \tau_{g_F} g(U+V,U+V).
\end{split}
\end{equation}
So, by \eqref{eq:timelike partialt} results:
\begin{cor} \label{cor:wec-1} Let $M=I _f\times F$ be a standard
static space-time with the metric $g=-f^2{\rm d}t^2 \oplus g_F.$
If the space-time $(M,g)$ satisfies the WEC, then $\tau_{g_F} \geq
0.$
\end{cor}

\begin{cor}\label{cor:wec-2} Suppose that $(F,g_F)$ is a Riemannian
manifold admitting at least one point at which the scalar curvature
is negative. Then, the WEC condition cannot be verified by any standard
static space-time
%of the form $M=I _f\times F$
with natural space
$(F,g_F)$.
\end{cor}

Recalling that on a Riemannian manifold $(F,g_F)$ of dimension $s$
the scalar curvature $\tau_{g_F} (p):= \sum_{j=1}^{s} {\Ric_{g_F}}
_p(e_j,e_j)$, where $\{e_j\}_{j=1}^s$ is an arbitrary orthonormal
basis for the tangent space $T_pF$ (see for instance
\cite{Berard1986}) and applying \eqref{eq: em tensor}, we obtain
the following couple of results when $\dim F \ge 2$.
\begin{thm} \label{thm:ec-w} Let $M=I _f\times F$ be a standard
static space-time with the metric $g=-f^2{\rm d}t^2 \oplus g_F$.
\begin{enumerate}
\item If $\mathcal{R}ic_{g_F}$ and $\mathcal{Q}^f_{g_F}$ are positive
(respectively, negative) semi-definite, then ${\rm
T}(\mathrm{w},\mathrm{w}) \geq 0$ (respectively, $\leq 0$) for any
causal vector $\mathrm{w} \in TM$.
\item If $(F,g_F)$ is Ricci flat, then for any $\mathrm{u} \in TI$
and $\mathrm{v} \in TF$, we have $8\pi \mathrm{T}(\mathrm{u} +
\mathrm{v},\mathrm{u} +
\mathrm{v})=\mathcal{Q}_{g_F}^f(\mathrm{v})$. Thus,
$\mathcal{Q}_{g_F}^f$ is positive semi-definite if and only if
${\rm T}(\mathrm{w},\mathrm{w}) \geq 0$ for any vector $\mathrm{w}
\in TM$.
\end{enumerate}
\end{thm}
%%
%%%%%%%%%%%%%%%%%%%%%%%%%%%%%%%%%%%
%%

\begin{thm} \label{thm:ec-w2} Let $M=I _f\times F$ be a standard
static space-time with the metric $g=-f^2{\rm d}t^2 \oplus g_F$.
\begin{enumerate}
%%%
%%%\item If $\displaystyle -\frac{1}{f}\mathcal{L}_{g_F}^\ast f =$
%%%$\displaystyle \mathcal{R}ic_F + \frac{1}{f}\mathcal{Q}^f_F$ are positive
%%%(respectively, negative) semi-definite and $\tau_F$ is nonnegative
%%%(respectively nonpositive), then ${\rm T}(\mathrm{w},\mathrm{w})
%%%\geq 0$ (respectively, $\leq 0$) for any causal vector $\mathrm{w}
%%%\in TM$.
%
\item If $\mathcal{L}_{g_F}^\ast f $
is negative (respectively, positive) semi-definite and
$\tau_{g_F}$ is nonnegative (respectively, nonpositive), then ${\rm
T}(\mathrm{w},\mathrm{w}) \geq 0$ (respectively, $\leq 0$) for any
causal vector $\mathrm{w} \in TM$.
\item If $\mathcal{L}_{g_F}^\ast f \equiv 0$,
%
%$\displaystyle \mathcal{R}ic_F + \frac{1}{f}\mathcal{Q}^f_F \equiv 0$,
%
then $\displaystyle 8\pi \mathrm{T}(\mathrm{w},\mathrm{w})=-
\frac{1}{2} \tau_{g_F} g(\mathrm{w},\mathrm{w})$ for any
$\mathrm{w} \in TM$. Thus, ${\rm T}(\mathrm{w},\mathrm{w}) \geq 0$
for any causal vector $\mathrm{w} \in TM$ if and only if
$\tau_{g_F} \geq 0$.
\end{enumerate}
\end{thm}

%%%%%%%%%%%%%%%%%%%%%%%%%%%%%%%%%%%

\begin{rem}\label{rem:KW obstructions} There have been strong and
intense studies about the topological significance of the scalar
curvature in Riemannian manifolds (see for instance among many
others \cite{Aubin98,BerardBergery1980,Kazdan1985,Kazdan2006})
since 1965s. In particular, Kazdan and Warner classified the compact
connected manifolds of dimension $\geq 3$ in three groups
\cite{Kazdan2006}:
\begin{enumerate}
    \item Those $N$ that admit a metric $h$ with scalar curvature $S_h \geq
    0$ (non identically $0$).

        \noindent Thus any function is the scalar curvature of some Riemannian metric.
    \item Those $N$ that admit no metric with positive scalar
    curvature, but do have a metric with $h \equiv 0$.

        \noindent So a function is the scalar curvature of some Riemannian
        metric if and only if is negative somewhere or is
        identically $0$.

    \item The other manifolds, so for any metric $h$, the scalar
    metric $S_h$ is negative somewhere.

        \noindent Thus a function is scalar curvature of some Riemannian metric
        if and only if is negative somewhere.
\end{enumerate}

\noindent Combining this result with the above \textit{Corollaries
\ref{cor:wec-1}} and \textit{\ref{cor:wec-2}}, we obtain that the
compact manifolds of the third Kazdan-Warner type cannot be the
natural Riemannian part of any standard static space-time
verifying the WEC.

\noindent Hence, for instance $(I \times (T^m \sharp T^m), -dt^2 +
h)$, where $h$ is a Riemannian metric on the connected sum of two
torus $T^m$ of dimension $m \ge 3$, never verifies the WEC.
Indeed, $T^m \sharp T^m$ belongs to the third Kazdan-Warner type
(see \cite{BerardBergery1980}).

\noindent We observe that topological obstructions for the problem
of prescribed scalar curvature of a Riemannian manifold does not
exist for noncompact and connected manifold, but the situation
changes if the completeness of the Riemannian manifold is required
(see \cite{BerardBergery1980}). Hence it would be interesting to
combine these results with the WEC.
\end{rem}

\begin{rem}\label{rem:DEC} Since the DEC implies the WEC, it is
clear by Corollary \ref{cor:wec-1} that $\tau_{g_F} \geq 0$ is
necessary for the DEC.

Furthermore, recalling that two causal tangent vectors
$\mathrm{w}_1,\mathrm{w}_2$ belong to the same time-cone iff
$g(\mathrm{w}_1,\mathrm{w}_2)\leq 0$ (see \cite[p. 143]{ON}) and
applying \eqref{eq:energy momentum ssst} it is easy to prove that:
\end{rem}

\begin{thm} \label{thm:skid} Let $M=I _f\times F$ be a standard
static space-time with the metric $g=-f^2{\rm d}t^2 \oplus g_F$.
\begin{enumerate}
\item If $\mathscr{L}_{g_F}^\ast f \equiv 0$, then $\displaystyle T = -\frac{1}{2} \tau_{g_F}g$.
Thus ${\rm T}(\mathrm{w}_1,\mathrm{w}_2) \geq 0$
for all causal vectors $\mathrm{w}_1,\mathrm{w}_2 \in
T(I _f\times F)$ in the same time-cone if and only if $\tau_{g_F} \geq 0$.
\item If $(F,g_F)$ is Ricci flat, then for any $\mathrm{u}_1,\mathrm{u}_2 \in TI$
and $\mathrm{v}_1,\mathrm{v}_2 \in TF$, we have $\displaystyle
8\pi
\mathrm{T}(\mathrm{u}_1+\mathrm{v}_1,\mathrm{u}_2+\mathrm{v}_2)=
\frac{1}{f}Q_{g_F}^f(\mathrm{v}_1,\mathrm{v}_2)$.
Thus, ${\rm T}(\mathrm{w}_1,\mathrm{w}_2) \geq 0$ for any causal
vectors $\mathrm{w}_1,\mathrm{w}_2 \in TM$ in the same time-cone
iff $Q_{g_F}^f \equiv 0$ iff $\mathscr{L}_{g_F}^\ast f \equiv 0$.
\end{enumerate}
\end{thm}

\begin{rem} To prove (2) in Theorem \ref{thm:skid}, it is useful to
note that for any $\mathrm{v_1} \in TF$, the vector given by
$\mathrm{w}_1:=r_1 \, f^{-1}
\,\sqrt{g_F(\mathrm{v_1},\mathrm{v_1})}
\partial_t + \mathrm{v_1} \in TM$ is causal iff $|r|\geq 1$. Moreover,
for any  $\mathrm{v_2} \in TF$, there exists $r_2 \in \mathbb{R}$
such that $\mathrm{w}_2 := r_2 \, f^{-1} \,
\sqrt{g_F(\mathrm{v_1},\mathrm{v_1})}
\partial_t + \mathrm{v_2} \in TM$ is causal (as for $\mathrm{v}_1$ it
is sufficient to take $|r_2|\geq 1$) and
$g(\mathrm{w}_1,\mathrm{w}_2)= - r_1 r_2
\sqrt{g_F(\mathrm{v_1},\mathrm{v_1})}
\sqrt{g_F(\mathrm{v_2},\mathrm{v_2})} +
g_F(\mathrm{v}_1,\mathrm{v}_2) \leq 0$, i.e., $\mathrm{w}_2$
belongs to the same time-cone as $\mathrm{w}_1$ does. Indeed, it
is sufficient to take $r_2$ such that $r_1 r_2 \geq \cos_{g_F}
\widehat{\mathrm{v}_1 \mathrm{v}_2}$.
\end{rem}

\begin{cor} \label{cor:ec-dec-skid} Let $M=I _f\times F$ be a standard
static space-time with the metric $g=-f^2{\rm d}t^2 \oplus g_F$.
If $\mathscr{L}_{g_F}^\ast f \equiv 0$ and $\tau_{g_F} \geq 0$,
then the DEC is satisfied.
\end{cor}

\begin{rem}\label{rem: kernel-csc}
Notice that by the polarization formula for a bilinear form,
\begin{equation}\label{eq:skid}
    \mathscr{L}_{g_F}^\ast f \equiv 0
\end{equation}
is equivalent to
\begin{equation}\label{eq:skid-q}
    \mathcal{L}_{g_F}^\ast f \equiv 0.
\end{equation}
So the hypothesis in Theorems \ref{thm:ec-w2}, \ref{thm:skid} and
Corollary \ref{cor:ec-dec-skid} are closely connected to the study
of the kernel of the operator $\mathscr{L}_{g_F}^\ast$ on suitable
Banach spaces. The latter question was studied by several authors,
see the references in \textit{Remark} \ref{rem:adjointsc}.

\noindent In particular, if $f$ is a solution of \eqref{eq:skid},
then
\begin{equation}\label{eq:traceadjoint-lsc}
    \displaystyle \trace_{g_F}\mathscr{L}_{g_F}^\ast f :=
-(s-1) \Delta_{g_F} f -{\tau}_{g_F} f = 0.
\end{equation}
Thus \eqref{eq:skid} takes the form
\begin{equation*}%\label{eq:adjoint-lsc}
    \displaystyle \mathscr{L}_{g_F}^\ast f = -f \, {\rm Ric}_{g_F} +
\frac{\tau_{g_F}}{s-1} f \, g_F +{\rm H}_{g_F}^f = 0.
\end{equation*}

\bigskip

In \cite[p. 228-230]{Fischer-Marsden1975-a} and \cite[p.
38-39]{Bourguignon1975} the authors proved that if $(F,g_F)$ is
compact and $f\not\equiv 0$ is a solution of \eqref{eq:skid}, then
the scalar curvature $\tau_{g_F}$ is constant (see
\cite{Corvino2000} also). So, by \eqref{eq:traceadjoint-lsc} and
the well known results about the spectrum of the Laplace-Beltrami
operator on a \textit{compact Riemannian manifold without
boundary} (see \cite{Berard1986}), there results:
%%%
%%%\begin{quote}\noindent if $(F,g_F)$ is \textit{compact} and $f \in C^\infty_{>0}
%%%(F)$, $f$ verifies \eqref{eq:skid} if and only if $f$ is constant
%%%and $(F,g_F)$ is Ricci flat.
%%%\end{quote}
%%%
%
\begin{equation}\label{eq:3.10compact}
\begin{split}
&\textrm{if }(F,g_F)\textrm{ is \textit{compact} and }f \in
C^\infty_{>0}
(F), \textrm{ then } f\textrm{ verifies }\eqref{eq:skid}\\
& \textrm{if and only if }f\textrm{ is constant and
}(F,g_F)\textrm{ is Ricci flat.}
\end{split}
\end{equation}

\bigskip

On the other hand, in \cite[Proposition 2.7]{Corvino2000} the
author proved that (see also \cite{Dobarro-Unal2004}):
\begin{equation}\label{eq:corvino}
\begin{split}
&f \not\equiv 0\textrm{ satisfies \eqref{eq:skid} if and only if the warped }\\
&\textrm{product metric }-f^2 dt^2+g_F\textrm{ is Einstein,}
\end{split}
\end{equation}
which is also easy to conclude from \eqref{eq:ricci-ssst}.

\bigskip

We remark also the importance of equation \eqref{eq:skid} in the
study of static Killing Initial Data (briefly \textit{static}
KIDs) in the recent articles of R. Beig, P. T. Chru\'{s}ciel, R.
Schoen, D. Pollack and F. Pacard
\cite{BeigChruscielSchoen2005,ChruscielPollack2008,ChruscielPacardPollack2008}.
%In their articles the authors do not make the hypothesis $f$ is
%strictly positive.
\end{rem}

Now, under additional hypothesis about the sign of the Ricci
curvature, we give a \2partial" extension of the
Bourguignon/Fischer-Marsden result mentioned above (\cite[p.
228-230]{Fischer-Marsden1975-a} and \cite{Bourguignon1975}) to the
case where the involved manifold is complete but noncompact.
Notice that the next three statements are supported by the
Liouville type results of P. Li and S.-T. Yau on complete but
noncompact Riemannian manifolds.

\begin{thm}\label{thm:einstein-corvino}Let $(F,g_F)$ be a complete
Riemannian manifold without boundary where $\dim F \ge 2$. Suppose
that the Ricci curvature of $(F,g_F)$ is nonnegative. Given $f \in
C^\infty_{>0} (F)$, $f$ is a solution of \eqref{eq:skid} if and only
if $f$ is constant and $(F,g_F)$ is Ricci-flat.
\end{thm}

\begin{proof}
This is a consequence of \eqref{eq:corvino} and \cite[Corollary
4.9]{Dobarro-Unal2004}.
\end{proof}

\begin{cor}
Let $(F,g_F)$ be a complete Riemannian manifold without boundary
where $\dim F \ge 2$. Suppose that $(F,g_F)$ is not Ricci-flat.
Then \eqref{eq:skid} admits a nonconstant strictly positive
solution only if the Ricci curvature of $(F,g_F)$ is negative
somewhere.
\end{cor}

\begin{cor}
Let $(F,g_F)$ be a complete Riemannian manifold without boundary
of $\dim F \ge 2$ which is \2either compact or complete with
nonnegative Ricci curvature" and also let $M=I _f\times F$ be a
standard static space-time with the metric $g=-f^2{\rm d}t^2
\oplus g_F$, where $f \in C^\infty_{>0}(F)$. Then, the following
properties are equivalent:
\begin{enumerate}
    \item the energy-momentum tensor of $M$ and $\tau_{g_F}$ are
    identically zero,
    \item $f$ is a solution of \eqref{eq:skid},
    \item $f$ is a constant and $M$ is Ricci-flat,
    \item $\Ric_{g_F} \equiv 0$ and $Q_{g_F}^f \equiv 0$,
    \item $f$ is a constant and $\Ric_{g_F} \equiv 0$,
\end{enumerate}
where $\Ric_{g_F}, Q_{g_F}^f$ and $\tau_{g_F}$ are as above. In
such cases, DEC (and consequently WEC) is trivially verified,
indeed the energy-momentum tensor of $M$ results identically zero.
\end{cor}

\begin{proof}
This is a consequence of \eqref{eq:3.10compact}, \textit{Theorem}
\ref{thm:einstein-corvino}, \eqref{eq:energy momentum ssst} and
\eqref{eq:ricci-ssst}.
\end{proof}

\section{Consequences and Examples}\label{sec:ConseqExamp}

In this section, unless otherwise stated, we will assume $\dim F =
s \geq 2$.

\subsection{Conformal Hyperbolicity}\label{subsec:ConfHyper}

Before we state our main results, we briefly recall the definition
the Lorentzian pseudo-distance on a Lorentzian manifold $(M,g_M)$
due to Markowitz and then recall some of its elementary properties
(see \cite{MM1} for further details).

\noindent Consider the open interval $(-1,1)$ furnished with the
Poincar\'{e} metric
\begin{equation*}\label{} {\rm
d}r^2_{(-1,1)}=\frac{{\rm d}u^2}{(1-u^2)^2}.
\end{equation*}
So the Poincar\'{e} distance between two points $u_0,u_1 \in
(-1,1)$ can be expressed as
\begin{equation*}
\rho(u_0,u_1) = \frac{1}{2} \left |\log \left(\frac{1+u_1}{1-u_1}
\frac{1-u_0}{1+u_0}\right) \right |.
\end{equation*}

\noindent Suppose $\gamma$ is a null pre-geodesic in $(M,g_M)$, i.e., $g(\gamma^\prime, \gamma^\prime)=0$ and $\gamma^{\prime
\prime}=\varphi \gamma^\prime,$ for some function $\varphi$. Then
there is a parameter, called the affine parameter, $r$ for which
the null pre-geodesic becomes a null geodesic. Indeed, $r$ is the
solution of the following ordinary differential equation
$\varphi=r^{\prime \prime}/r.$ In this case, a projective
parameter $p$ is defined to be a solution of
\begin{equation*}\label{}
\{p \,; r\}:=\frac{1}{2}\frac{p^{\prime \prime
\prime}}{p^\prime}-\frac{3}{4} \left(\frac{p^{\prime
\prime}}{p^\prime}\right)^2=-\frac{1}{n-2}
{\Ric_M}(\gamma^\prime(r),\gamma^\prime(r)),
\end{equation*}
where $r$ is an affine parameter for a null pre-geodesic of the
form given above, $\{p \,; r\}$ is the so called Schwarzian
derivative of $p$ respect to $r$ (see \cite{Sasaki1999}) and
$\Ric_M$ is the Ricci tensor of the Lorentzian manifold $(M,g_M)$.
The parameter $r$ is independent of the affine parameter along
$\gamma$.

\noindent A chain of null geodesic segments joining $p$ to $q$ is
\begin{itemize}
\item a sequence of points $p=p_0,p_1,\dots,p_k=q$ in $M,$
\item pairs of points $(a_1,b_1), \dots, (a_k,b_k)$ in $(-1,1)$
and
\item projective maps (i.e., a projective map is simply a null
geodesics with the projective parameter as the natural parameter)
$f_1,\dots,$ $f_k$ from $(-1,1)$ into $M$ such that
$f_i(a_i)=p_{i-1}$ and $f_i(b_i)=p_i$ for $i=1,\cdots,k.$
\end{itemize}
The length of such a chain is defined as $L(\alpha)= \sum_{i=1}^k
\rho(a_i,b_i)$.

\noindent By combining all the ingredients defined above, Markowitz
defines his version of intrinsic Lorentzian pseudo-distance
$d_M \colon M \times M \to [0,\infty)$ as
\begin{equation*}\label{} d_M(p,q):=\inf_{\alpha}L(\alpha),
\end{equation*}
where the infimum is taken over all the chains of null geodesic
segments $\alpha$ joining $p$ to $q$. Note that $d_M$ is indeed
a pseudo-distance. The Lorentzian manifold $(M,g_M)$ is called
\textit{conformally hyperbolic} when the Lorentzian
pseudo-distance $d_M$ is a true distance.

We now reproduce the statements of three theorems from \cite{MM1}
which will be useful later.

\medskip

\noindent \cite[\textsc{Theorem 5.1}]{MM1} Let $(M,g)$ be a null
geodesically complete Lorentz manifold. If $(M,g)$ satisfies the
curvature condition ${\rm Ric}_M(\mathrm{x},\mathrm{x}) \leq 0$
for all null vectors $\mathrm{x}$, then it has a trivial
Lorentzian pseudo-distance, i.e., $d_M \equiv 0.$

\medskip

\noindent \cite[\textsc{Theorem 5.8}]{MM1} Let $(M,g)$ be an
$n(\geq 3)$-dimensional Lorentzian manifold. If $(M,g)$ satisfies
the NCC and the null generic condition, briefly NGC, (i.e., ${\rm
Ric}(\gamma^\prime,\gamma^\prime) \neq 0,$ for at least one point
of each inextendible null geodesic $\gamma$) then, it is
conformally hyperbolic.

\medskip

\noindent \cite[\textsc{Theorem 7.1}]{MM1} The group of conformal
automorphisms of a conformally hyperbolic Lorentzian manifold
$(M,g)$ has a compact isotropy group at each point $p.$

\medskip

We will recall some examples in \cite{MM1}.
\begin{itemize}
\item Complete Einstein space-times (in particular, Minkowski,
de Sitter and the anti-de Sitter space-times) have trivial
Lorentzian pseudo-distances because of Theorem 5.1 of \cite{MM1}.
\item The Einstein static universe has also trivial Lorentzian
pseudo-distance since the space-times in the previous item can be
conformally imbedded in the Einstein static universe.
\item A Robertson-Walker space-time (i.e., an isotropic
homogeneous space-time) is conformally hyperbolic due to
\cite[Theorem 5.9]{MM1}.
\item The Einstein-de Sitter space $M$ is conformally hyperbolic
and for the null separated points $p$ and $q,$ the Lorentzian
pseudo-distance is given by $$d_M(p,q) = \frac{5}{4} \log
r(p,q),$$ where $r(p,q)$ denotes cosmological frequency ratio (see
\cite{MM2} for the explicit computation). The conformal distance
between two causally related events in the Einstein-de Sitter
space is given as $$d_M(x,y)=\frac{1}{2} \left| \log
\frac{\tau(x)}{\tau(y)}\right|,$$ where $M=(0,\infty) \times
_{t^{2/3}} \mathbb R^3$ with $g=-{\rm d}t^2 \oplus t^{4/3} {\rm
d}\sigma^2$ and $\tau=3t^{1/3}$ (see \cite[Theorem 5]{MM2}).
\end{itemize}

\bigskip

By applying Theorem \ref{thm:ec-t} and the previously restated results
of Markowitz, we obtain the theorems that follow.

\begin{thm} \label{main-12}
Let $M=\mathbb R _f\times F$ be a standard static space-time with
the metric $g=-f^2{\rm d}t^2 \oplus g_F.$ Suppose that
$\mathcal{L}_{g_F}^\ast f$ is positive semi-definite (this condition
is satisfied for example, if $\mathcal{R}ic_{g_F}$ and $\mathcal{Q}_{g_F}^f$
are negative semi-definite). If at least one of the following conditions is
verified
\begin{enumerate}
\item $(F,g_F)$ is compact,
\item $(F,g_F)$ is complete and  $0 < \inf f $,
\end{enumerate}
then the Lorentzian pseudo-distance $d_M$ on the standard static
space-time $(M,g)$ is trivial, i.e., $d_M \equiv 0.$
\end{thm}

\begin{proof}
By \textit{Theorem \ref{thm:ec-t}} and \eqref{eq:ricci-ssst}, the
hypothesis for $\mathcal{L}_{g_F}^\ast f$
%
%$\mathcal{R}ic_F$ and $\mathcal{Q}^f_F$
%
being positive semi-definite implies %imply
the RNCC. In both cases the null geodesic completeness of the
underlying standard static space-times is a consequence of
\cite[Theorem 3.12]{Allison1988-b}. So, by \cite[Theorem
5.1]{MM1}, $\mathbb R _f\times F$ has a trivial Lorentzian
pseudo-distance, i.e., $d_M \equiv 0$.
\end{proof}

\begin{thm} \label{main-3}
Let $M=I _f\times F$ be a standard static space-time with the
metric $g=-f^2{\rm d}t^2 \oplus g_F$. Then, $(M,g)$ is conformally
hyperbolic if at least one of the following properties is
satisfied:
\begin{enumerate}
    \item $\mathcal{L}_{g_F}^\ast f$ is negative semi-definite
and the NGC is verified,
    \item $\mathcal{R}ic_{g_F}$ is positive semi-definite and $\mathcal{Q}_{g_F}^f$
is positive definite.
%%%%%
%%%%%    \item $\mathcal{R}ic_F$ is positive definite and $\mathcal{Q}^f_F$
%%%%%is semi-positive definite.
%%%%%
\end{enumerate}
\end{thm}

%%%
%%%\begin{thm} \label{1main-3}
%%%Let $M=I _f\times F$ be a standard static space-time with the
%%%metric $g=-f^2{\rm d}t^2 \oplus g_F$. Suppose that
%%%$\mathcal{R}ic_F$ is positive semi-definite and $\mathcal{Q}^f_F$
%%%is positive definite. Then the standard static space-time $(M,g)$
%%%is conformally hyperbolic.
%%%\end{thm}
%%%

\begin{proof}
As above, \textit{Theorem \ref{thm:ec-t}} and \eqref{eq:ricci-ssst},
together with the negative semi-definiteness of
$\mathcal{L}_{g_F}^\ast f$ or positive semi-definiteness of
\2$\mathcal{R}ic_{g_F}$ and $\mathcal{Q}_{g_F}^f$" imply the NCC.
In $(2)$, the positive definiteness of $\mathcal{Q}_{g_F}^f$
infers the NGC. Hence, in both cases all the assumptions of
\cite[Theorem 5.8]{MM1} are verified and $\mathbb R _f\times F$
results conformally hyperbolic.
\end{proof}

\subsection{Conjugate Points}\label{subsec:ConjugatePoints}

We now state some results relating the conformal hyperbolicity and
causal conjugate points of a standard static space-time by using
\cite{Blgl,Bcclc,Bsild,CE} and also \cite{BEE}. In \cite[Theorem
2.3]{CE}, it was shown that if the line integral of the Ricci
tensor along a complete causal geodesic in a Lorentzian manifold
is positive, then the complete causal geodesic contains a pair of
conjugate points.

Assume that $\gamma=(\alpha, \beta)$ is a complete causal geodesic
in a standard static space-time of the form $M=I _f\times F$ with
the metric $g=-f^2{\rm d}t^2 \oplus g_F.$ Then by using
$g(\gamma^\prime, \gamma^\prime) \leq 0$ and \eqref{eq:ricci-ssst}
we have,
%%%
%%%\begin{equation*}\label{} {\rm Ric}(\gamma^\prime,
%%%\gamma^\prime)= {\rm Ric}_F(\beta^\prime, \beta^\prime) +
%%%\frac{1}{f}Q^f_F(\beta^\prime,\beta^\prime)-
%%%\underbrace{g(\gamma^\prime,\gamma^\prime)}_{\leq 0}
%%%\frac{1}{f}\Delta_F f.
%%%\end{equation*}
%%%
%%%
\begin{equation*}
{\rm Ric}(\gamma^\prime, \gamma^\prime)=\displaystyle -\frac{1}{f}
\mathscr{L}_{g_F}^\ast f (\beta^\prime, \beta^\prime)
   -
\underbrace{g(\gamma^\prime,\gamma^\prime)}_{\leq 0}
\frac{1}{f}\Delta_{g_F} f.
\end{equation*}
%%%
We can easily state the following existence result for conjugate
points of complete causal geodesics in a conformally hyperbolic
standard static space-time by {\it Theorem \ref{main-3}} and
\cite[Theorem 2.3]{CE}.

%%%
%%%\begin{cor} \label{cor-1} Let $M=I _f\times F$ be a standard
%%%static space-time with the metric $g=-f^2{\rm d}t^2 \oplus g_F$.
%%%Suppose that $\mathcal{R}ic_F$ is positive semi-definite. If
%%%$\mathcal{Q}^f_F$ is positive definite, then $(M,g)$ is
%%%conformally hyperbolic and any complete causal geodesic in $(M,g)$
%%%has a pair of conjugate points.
%%%\end{cor}
%%%

%%%
\begin{cor} \label{cor-1} Let $M=I _f\times F$ be a standard
static space-time with the metric $g=-f^2{\rm d}t^2 \oplus g_F$.
Then, $(M,g)$ is conformally hyperbolic and any complete causal
geodesic in $(M,g)$ has a pair of conjugate points, if at least
one of the following properties is satisfied:
\begin{enumerate}
    \item $\mathcal{L}_{g_F}^\ast f$ is negative definite, $f$ is
    subharmonic %($\Delta_F \geq 0$)
and the NGC is verified,
    \item $\mathcal{R}ic_{g_F}$ is positive semi-definite and $\mathcal{Q}_{g_F}^f$
is positive definite.
\end{enumerate}
\end{cor}
%%%

We will give an existence result for conjugate points of time-like
geodesics in a standard static space-time, which will also be
conformally hyperbolic because of \textit{Theorem \ref{main-3}} if
one replaces the positive semi-definiteness of
$\mathcal{Q}_{g_F}^f$ by the positive definiteness .

\noindent In the next corollary, $\mathbf{L}$ denotes the usual
time-like Lorentzian length and ${\rm diam}_\mathbf{L}$ denotes
the corresponding time-like diameter (see \cite[Chapters 4 and
11]{BEE}).

\begin{cor} \label{cor-3} Let $M=I _f\times F$ be a standard
static space-time with the metric $g=-f^2{\rm d}t^2 \oplus g_F$.
Suppose that $\mathcal{R}ic_{g_F}$ and $\mathcal{Q}_{g_F}^f$ are
positive semi-definite. If there exists a constant $c$ such that
$\displaystyle\frac{1}{f}\Delta_{g_F} f \geq c >0$, then
\begin{enumerate}
%\item $(M,g)$ is conformally
%hyperbolic,
%
\item any time-like geodesic $\gamma \colon [r_1,r_2] \to M$ in $(M,g)$ with $\displaystyle{\mathbf
L}(\gamma) \geq \pi \sqrt{\frac{n-1}{c}}$ has a pair of conjugate
points,

\item for any time-like geodesic $\gamma \colon [r_1,r_2] \to M$ in
 $(M,g)$ with $\mathbf{L}(\gamma)
> \displaystyle{\pi\sqrt{\frac{n-1}{c}}}$,  $r=r_1$ is conjugate along $\gamma$ to some
$r_0 \in (r_1,r_2),$ and consequently $\gamma$ is not maximal,

\item if $I=\mathbb{R}$, $(F,g_F)$ is complete and $\sup f<\infty$,
 then $$\displaystyle{\rm diam}_\mathbf{L}(M,g) \leq
 \pi\sqrt{\frac{n-1}{c}}.$$
\end{enumerate}
\end{cor}

\begin{proof}
First of all we observe that \eqref{eq:ricci-ssst} implies
\begin{equation}\label{eq:BEEp404}
\Ric(\mathrm{u} + \mathrm{v},\mathrm{u} + \mathrm{v}) \geq
\frac{\Delta_{g_F} f}{f} \geq c
%=(n-1)\frac{c}{n-1}
> 0,
\end{equation}
for any unit time-like tangent vector $\mathrm{u} + \mathrm{v}$ on
$M$, where $\mathrm{u} \in TI$ and $\mathrm{v} \in TF$.
%
%Proposition 11.7 and 11.8 and
%also Theorem 11.9 in \cite{BEE} (see also
%\cite{Blgl,Bcclc,Bsild}).
%
\begin{enumerate}
    \item It is an immediate consequence of \eqref{eq:BEEp404}
    and \cite[Proposition 11.7]{BEE}.
    \item It is an immediate consequence of \eqref{eq:BEEp404}
    and \cite[Proposition 11.8]{BEE}.
    \item In \cite[Corollary 3.17]{Allison1985} the author proves
    ``if $(F,g_F)$ is a complete Riemannian manifold and
    $\sup f<\infty$, then $\mathbb{R} _f\times F$ is globally
    hyperbolic". So by applying \eqref{eq:BEEp404} and
    \cite[Theorem 11.9]{BEE}, (3) is obtained.
\end{enumerate}
\end{proof}

\begin{rem} $ $

\begin{itemize}
\item [{\bf (i)}]
Notice that Corollary \ref{cor-3} holds if the positive
semi-definiteness of $\mathcal{R}ic_{g_F}$ and
$\mathcal{Q}_{g_F}^f$ is replaced by the negative
semi-definiteness of $\mathcal{L}_{g_F}^\ast f$.
\item [{\bf (ii)}]
On the other hand, by using \cite[\textsc{Theorem 7.1}]{MM1}, one
can also deduce that the group of conformal automorphisms of the
the underlying standard static space-time has a compact isotropy
group at each point $p$ when the hypothesis in {\it Theorem
\ref{main-3}} or {\it Corollary \ref{cor-1}} is verified.
\end{itemize}
\end{rem}

\begin{rem}\label{rem:compact fiber 1} We observe that in
{\it Theorems \ref{thm:ec-t}} and {\it \ref{main-3}} and {\it
Corollaries \ref{cor-1}
%, \ref{cor-2}
} and {\it \ref{cor-3}}, when the semi-definite positiveness of
$\mathcal{Q}_{g_F}^f$ is considered, necessarily $\Delta_{g_F} f
\ge 0$ (i.e., the function $f$ is sub-harmonic) if $\dim F \geq
2$ (see Remark \ref{rem:quadratic form}). Consequently, in all
these statements, if the Riemannian manifold $(F,g_F)$ is compact
of dimension at least $2$, then $f$ turns out to be a positive
constant (see \textit{Remark \ref{rem:quadratic form}}).

\noindent Thus, the case of a noncompact Riemannian manifold
$(F,g_F)$ is particularly relevant. In order to show that the
considered assumptions are nonempty (see the discussion about the
existence sub-harmonic on complete Riemannian manifolds in
\textit{Subsection \ref{subsec:liouville}}), we will provide the
following simple example:
$$\mathbb{R}^s \hbox{ with the usual Euclidean metric } g_0$$
$$M=I \times \Omega , \Omega \subseteq \mathbb{R}^s \hbox{ open in }
\mathbb{R}^s \hbox{ with the induced metric } g_0$$
$$f:\Omega \rightarrow \mathbb{R}; f(x)=
\displaystyle\frac{1}{2}|x|_{g_0}^2+\epsilon > 0, \epsilon
>0$$
$$g=-f^2 {\rm d}t^2 + g_0$$
$${\rm H}_{\Omega}^f=  \, g_0
%{\rm Id}_{\mathbb{R}^s}
$$
$$\Delta_{\Omega} f\equiv  s >0$$
$$\Ric_{\Omega}\equiv 0$$

\noindent Thus, the assumptions of {\it Theorems \ref{thm:ec-t}}
and {\it \ref{main-3}} and {\it Corollary \ref{cor-1}}
%{\it Corollaries \ref{cor-1}} and {\it \ref{cor-2}}
are verified. Indeed,
\begin{itemize}
    \item[{\bf (i)}] For any $x \in \Omega,$ the
    quadratic form $-\mathcal{L}_{g_0}^\ast f = \mathcal{Q}^f_{\Omega}
    =(s-1) |\cdot|^2_{g_{0}}
    %=(s-1) {\rm Id}_{\mathbb{R}^s}
    $
    is positive definite on $T_x \Omega=\mathbb{R}^s$.
    \item[{\bf (ii)}] $\Ric_{\Omega}\equiv 0$.
\end{itemize}
If we suppose that $\Omega$ is bounded, then the assumptions of
\textit{Corollary \ref{cor-3}} are verified with an exception of
item $3$ (i.e., the completeness of $(F,g_F)$). In this case,
there exists a constant $c>0$ such that
$$\displaystyle
\displaystyle\frac{1}{f}\Delta_\Omega
f=\frac{s}{\displaystyle\frac{1}{2}|x|^2+\epsilon} \ge c
>0.$$
\end{rem}

\subsection{Some special examples} \label{subsec:special examples}
Now, we will concentrate our attention on some special families of
metrics that allow an easier control of the sign of
$\mathcal{Q}_{g_F}^f$. For these metrics, there is a strong
relation between the Riemannian metric $g_F$ and the function $f
\in C^\infty_{>0}(F)$.

\medskip

``\emph{Concircular scalar fields}". We follow the terminology
used in \cite{Tashiro65}, i.e., we will call a nonconstant scalar
field $u$ on $F$ as a \textit{concircular scalar field} if it
satisfies the equation
\begin{equation}\label{eq:concircular fields 1}
    {\rm H}_{g_F}^u=\phi g_F,
\end{equation}
where $\phi$ is a scalar field called the \textit{characteristic
function} of $u$.

Notice that in this situation, taking the $g_F-$trace, we obtain
\begin{equation}\label{eq:concircular fields 2}
    \Delta_{g_F} u = \phi s
\end{equation}
and as a consequence,
\begin{equation}\label{eq:concircular fields 3}
    Q_{g_F}^u  = (s-1)\phi g_F.
\end{equation}
Thus, being $\dim F \geq 2$, the sign of $\mathcal{Q}_{g_F}^u$ is
determined by $\sign \phi$.

We observe that equation \eqref{eq:concircular fields 2} admits a
nonconstant solution with a smooth function $\phi$ of constant
sign only if the Riemannian manifold $(F,g_F)$ is noncompact.
Thus, there exists no concircular scalar field with a
characteristic function of constant sign on a compact Riemannian
manifold.
%\end{rem}

Thus, applying \textit{Theorem \ref{main-12}} and \textit{Theorem
\ref{main-3}}, respectively,  we state the results that follow.
\begin{cor}\label{cor:concircular fields 1}
Let $(F,g_F)$ be such that $\mathcal{R}ic_{g_F}$ is negative
semi-definite. Suppose that $(F,g_F)$ admits a positive
concircular scalar field $u$ with a nonpositive characteristic
function $\phi$. If $(F,g_F)$ is complete and $0 < \inf u $, then
the Lorentzian pseudo-distance $d_M$ on the standard static
space-time $M=I _u\times F$ with the metric $g=-u^2{\rm d}t^2
\oplus g_F$ is trivial, i.e., $d_M \equiv 0.$
\end{cor}
\begin{cor}\label{cor:concircular fields 2}
Let $(F,g_F)$ be such that $\mathcal{R}ic_{g_F}$ is negative
semi-definite. If $(F,g_F)$ admits a positive concircular scalar
field $u$ with a positive characteristic function $\phi$, then the
standard static space-time $M=I _u\times F$ with the metric
$g=-u^2{\rm d}t^2 \oplus g_F$ is conformally hyperbolic.
\end{cor}

\medskip

``\emph{Hessian manifolds}". This type of manifolds are closely
related to the previously defined concircular scalar fields. We
begin by recalling the definition of Hessian manifolds (see some
recent articles about this kind of manifolds such as
\cite{{CV},{D},{Shima1995},{SY},{Totaro2004}}, and in particular,
the recent book by H. Shima \cite{Shima2007}).

First of all, we remark a notational matter: we say that a
connection $D$ on a manifold $N$ is \emph{S-flat} if its
torsion and curvature tensor vanish identically. We also
say that a manifold $N$ endowed with an \emph{S-flat}
connection $D$ is an \emph{S-flat} manifold and we denote
it by $(N,D)$. Notice that in \cite{Shima2007}, \emph{S-flat}
manifolds are called flat. On the other hand, according to our
convention, the concept of flatness only means that the curvature
tensor vanishes identically.

\smallskip

A Riemannian metric $h$ on an \emph{S-flat} manifold $(N,D)$ is
said to be a locally (respectively, globally) Hessian metric if
$h$ is locally (respectively, globally) expressed by the Hessian,
i.e., $h=Ddu$, where $u$ is a local (respectively, global) smooth
function on $N$. Such pair $(D,h)$ is called a locally
(respectively, globally) Hessian structure on $N$, and $u$ is said
to be a local (respectively, global) potential of $(D,h)$. A
manifold $N$ provided with a locally (respectively, globally)
Hessian structure is called a locally (respectively, globally)
Hessian manifold and is denoted by $(N,D,h)$.

\noindent If it is necessary, we will say that $N$ is $u-$globally
Hessian to indicate the specific function $u \in C^\infty(N)$ such
that $h=Ddu$ on the entire manifold $N$.

\noindent It is clear that a global Hessian manifold is also a
local one.

\noindent In \cite{Shima2007}, it is proved that if a Riemannian
manifold $(N,h)$ is flat with respect to the Levi-Civita
connection $\nabla$, then it is also a locally Hessian manifold
with the Hessian structure $(\nabla, h)$. Notice that for the
Levi-Civita connection both concepts, i.e., flatness and
\emph{S-flatness} coincide.

\noindent Moreover, in \cite{Shima2007} several important examples
of local and global Hessian structures are provided. Here, we
mention more classical and well known one, namely $\mathbb{R}^s$
furnished with the usual canonical flat connection. More
explicitly, $\mathbb{R}^s$ is a flat $u-$globally Hessian manifold
where $u(x)= \displaystyle\frac{1}{2}|x|_{g_0}^2$, where $g_0$ is
the canonical Euclidean metric on $\mathbb{R}^s$.

\noindent Let $(N,h)$ be a Riemannian manifold with the
Levi-Civita connection $\nabla$. Then, $(\nabla,h)$ is a globally
Hessian structure on $N$ with positive global potential if and
only if $(N,h)$ is flat and admits a positive concircular scalar
field with characteristic function $\phi\equiv 1$. In this case,
the positive concircular scalar field can be taken as the positive
global potential, and viceversa. Notice that when we say
$(\nabla,h)$ is a globally Hessian structure on $N$ with positive
global potential, this particularly implies that the global
potential is nonconstant (since $h$ is a metric).
Thus, applying \textit{Corollary \ref{cor:concircular fields 2}}
results:
\begin{cor}\label{cor:hessian fiber}
Let $(F,g_F)$ be a $u-$globally Hessian manifold
%of dimension $s \ge 2$
with the Hessian structure $(\nabla^{g_F},g_F)$, where the global
potential $u \in C^\infty_{>0}(F)$ and $\nabla^{g_F}$ denotes the
Levi-Civita connection of $g_F$. Then, the standard static
space-time of the form $M=I _u\times F$ with the metric
$g=-u^2{\rm d}t^2 \oplus g_F$ is conformally hyperbolic.
\end{cor}

\begin{rem}
\label{rem:general rem conf hyp} We will note some facts related
to the previous statements and examples of this section.
\begin{itemize}
\item [{\bf (i)}]Let now $(D,h)$ be a locally Hessian structure on
a manifold $N$ of dimension $s$ and $\nabla $ be the Levi-Civita
connection associated to $h$. In \cite{SY}, necessary and
sufficient conditions are given to have $\nabla = D,$ when $N$ is
a compact manifold without boundary.

\noindent Note that in this case, i.e., $h=Ddu=\nabla d u$, we
necessarily have $\Delta u = s.$ Indeed, it is sufficient to take
the $h-$trace of the both sides.

\noindent Furthermore, if the Hessian structure is $u-$global, then
$u$ is a function with Laplacian of constant sign on a compact
Riemannian manifold. Hence, $u$ is a constant (see \textit{Remark
\ref{rem:quadratic form}}).

\noindent In particular, this shows that the relevant cases for
\textit{Corollary \ref{cor:hessian fiber}} are those in which
$(F,g_F)$ is assumed to be noncompact. Indeed, if it is compact,
then the assumptions of the last corollary are not verified.

\noindent One simple example of this situation is the Riemannian
manifold (i.e., globally Hessian manifold) $(F,g_F)=(\Omega,g_0)$
where $\Omega$ is an open subset of $\mathbb{R}^s$ considered in
\textit{Remark \ref{rem:compact fiber 1}}.

\noindent If we further add the completeness assumption for $(F,g_F)$
to the hypothesis of \textit{Corollary \ref{cor:hessian fiber}}, then
$(F,g_F)$ becomes an Euclidean space. Indeed, this is a consequence
(for instance) of \cite[Theorem 2 (I,B)]{Tashiro65}.

\item [{\bf (ii)}] If $(F,g_F)$ is a complete Riemannian
manifold of dimension $1$ and $u \in C^\infty_{>0}(F)$ verifying
$0 < \inf u $, then the Lorentzian pseudo-distance $d_M$ on a
standard static space-time of the form $M=I _u\times F$ with
metric $g=-u^2{\rm d}t^2 \oplus g_F$ is trivial i.e., $d_M \equiv
0.$ Indeed, $\dim F = s = 1$ implies that $Q_{g_F}^u \equiv 0$ (by
the definition of $Q_{g_F}^u$), so we can directly apply
\textit{Theorem \ref{main-12}}.

\noindent By means of analogous arguments, the hypothesis stating
that ``\textit{$u$ is a concircular scalar field}" in \textit{Corollary
\ref{cor:concircular fields 1}} turns out to be irrelevant when
the dimension of the fiber is $1$.

\item [{\bf (iii)}] By considering the relation \eqref{eq:concircular fields 3},
we can notice that relevant examples for \textit{Corollaries
\ref{cor:concircular fields 1}} and \textit{\ref{cor:concircular
fields 2}} correspond to the case of having a Riemannian part of
dimension at least $2$.

\noindent For instance, in $\mathbb{R}^n$ with the usual metric
$g_0 ={\rm Id}_{\mathbb{R}^n}$, the existence of a concircular
scalar field, i.e., ${\rm H}^u_{\mathbb{R}^n}=\phi g_0$, implies
%(studying the PDE)
that $u(x)= \sum_{i=1}^n \xi_i(x_i)$
%(considering the mixed derivatives)
and hence,
%(now by the non mixed derivatives)
$\partial_{ii}\xi_i(x_i) = \partial_{jj}\xi_j(x_j) = \phi(x)$ for
any $i \in \{1,\cdots,n\}-\{j\}$. Thus, $\phi = \phi_0$ is
constant and $\xi_i(x_i)= a_i x_i^2 + b_i x_i + c_i$, where
$\phi_0=2a_i$ for any $i$ with the coefficients $a_i, b_i, c_i \in
\mathbb R.$

%\noindent As one can easily notice, these examples and those in
%\textit{Remark \ref{rem:compact fiber 1}} are closely related.

%%
%%CHW_060623.dvi
%%
\noindent Finally, we will mention some simple examples where the
suppositions of Corollaries \ref{cor:concircular fields 2} and
\ref{cor:concircular fields 1} are not satisfied. Let
$(\mathbb{R}^2,g_0)$ be the usual Euclidean plane and $\Omega $
be an \emph{open} subset of $\mathbb{R}^2$. Then $(\Omega,g_0)$ is a
Ricci-flat Riemannian manifold. Therefore:
\begin{enumerate}
 \item If $\Omega = \mathbb{R}^2$, then the hypothesis of
    \textit{Corollary \ref{cor:concircular fields 1}} cannot be
    verified. Indeed $(\mathbb{R}^2,g_0)$ is complete,
    but does not admit a positive concircular scalar field $u$ with
    a nonpositive characteristic function $\phi$ such that $0 < \inf u$.
    This follows from \eqref{eq:concircular fields 2} and either direct
    calculus or the classical Liouville nonexistence theorems for
    nonconstant superharmonic functions bounded from below on
    $(\mathbb{R}^2,g_0)$ (see \cite{Protter-Wienberger1984,Serrin-Zou2002}).
 \item If $\Omega = \{x \in \mathbb{R}^2: x_1>0, x_2>0\}$, then the hypothesis
    of \textit{Corollary \ref{cor:concircular fields 2}} is verified. For
    example, take $u(x_1,x_2)=x_1^2+x_2^2+x_1+x_2+1$ as a concircular scalar
    field with positive characteristic function $2$.
 \item For any $\Omega$, the hypothesis of \textit{Corollary
    \ref{cor:concircular fields 2}} is satisfied. For instance, take
    $u(x_1,x_2)=x_1^2+x_2^2+1$ as a concircular scalar field with positive
    characteristic function $2$. Notice that this concircular scalar field is
    different from the one given in (2).
\end{enumerate}
\end{itemize}
\end{rem}

\section{Conclusions}\label{sec:conclusions}

Our first goal is to obtain a family of necessary and/or
sufficient conditions for a set of energy and convergence
conditions on standard static space-times. We state a family of
results establishing a connection between the strong energy
condition and a family of Liouville type results for subharmonic
functions on the Riemannian part of a standard static space-time
too. We also obtain a set of results about the weak and dominant
energy conditions. All these results are based on suitable
hypothesis for the definiteness of the quadratic forms associated
to the Ricci tensor and the tensors $\mathscr{L}_{g_F}^\ast f$ and
$Q_{g_F}^f$ defined in \textit{Section \ref{sec:preliminaries}}.
Similar to the independent studies of Bourguignon, Fischer-Marsden or
Lafontaine, we observe the importance of the kernel of the operator
$\mathscr{L}_{g_F}^\ast$, not only in the study of standard static
space-times, but also in Riemannian geometry, more specifically in
the study of critical points of the scalar curvature map in any dimension.
About this matter, taking into account the sign of the Ricci curvature
and applying Liouville type results, we provide a statement about the
existence of positive functions in the kernel of the operator
$\mathscr{L}_{g_F}^\ast $ where the involved manifold is complete but
noncompact.

The second goal is to apply the previous results together with
several results of M. J. Markowitz and D. E. Allison to obtain
sufficient conditions for studying the conformal hyperbolicity
and the existence of conjugate points on causal geodesics on a
standard static space-time.

Finally, we show some examples and results relating the tensor
$Q_F^f$, conformal hyperbolicity, concircular scalar fields and
Hessian manifolds.

\pdfbookmark[1]{Bibliography}{bib}

\end{document}